\documentclass[12pt]{amsart}
\usepackage{amssymb}
\usepackage{amsmath, amscd}
\usepackage{amsthm}
\usepackage{amsmath}
\usepackage{comment}
\usepackage[table,xcdraw]{xcolor}
\usepackage{ulem}
\usepackage{tikz}
\usepackage{float}
\usepackage{graphicx}
\usepackage{circuitikz}
\usetikzlibrary{positioning}
\usepackage[colorlinks=true, linkcolor=blue,urlcolor=blue]{hyperref}
\usepackage{nicematrix}

\newtheorem{theorem}{Theorem}[section]

\newtheorem{proposition}[theorem]{Proposition}
\newtheorem{lemma}[theorem]{Lemma}

\newtheorem{corollary}[theorem]{Corollary}
\theoremstyle{definition}

\newtheorem{example}[theorem]{Example}
\newtheorem{definition}[theorem]{Definition}

\newtheorem{remark}[theorem]{Remark}

\newcommand{\A}{\mathcal{A}}

\def\a{\alpha}

\date{\today}

\begin{document}
	
\author[A. Moussavi]{Ahmad Moussavi}
\address{Department of Mathematics, Tarbiat Modares University, 14115-111 Tehran Jalal AleAhmad Nasr, Iran}
\email{moussavi.a@modares.ac.ir; moussavi.a@gmail.com}

\author[P. Danchev]{Peter Danchev}
\address{Institute of Mathematics and Informatics, Bulgarian Academy of Sciences, 1113 Sofia, Bulgaria}
\email{danchev@math.bas.bg; pvdanchev@yahoo.com}

\author[A. Javan]{Arash Javan}
\address{Department of Mathematics, Tarbiat Modares University, 14115-111 Tehran Jalal AleAhmad Nasr, Iran}
\email{a.darajavan@modares.ac.ir; a.darajavan@gmail.com}

\author[O. Hasanzadeh]{Omid Hasanzadeh}
\address{Department of Mathematics, Tarbiat Modares University, 14115-111 Tehran Jalal AleAhmad Nasr, Iran}
\email{o.hasanzade@modares.ac.ir; hasanzadeomiid@gmail.com}

\title[generalized $\sqrt{J}$-fine rings]{Expanding generalized fine rings}

\maketitle

\begin{abstract}
We introduce and study the so-called {\it generalized $\sqrt{J}$-fine rings}, where every element outside the Jacobson radical is the sum of a unit and an element from the set $\sqrt{J(R)} := \{ x \in R : x^{n} \in J(R) \text{ for some } n \ge 1 \}$. This commonly extends the notions of {\it fine} and {\it generalized fine rings} defined, respectively, by C\u{a}lug\u{a}reanu-Lam (J. Algebra \& Appl., 2016) and Zhou (J. Algebra \& Appl., 2022).

Specifically, we prove that this class is closed under full matrix rings of any size, as well as we completely characterize when group rings over locally finite groups are generalized $\sqrt{J}$-fine. We also show that every such ring is 2-clean, thus properly placing it between generalized fine rings and 2-clean rings. Several examples are also provided to illustrate the complicated behavior of the introduced concept and its numerous boundaries.
\end{abstract}

\section{Introduction and Principal Tools}

In ring theory, one of the most natural questions we can ask is the following: \textit{How do the addition and multiplication operations of a ring interact with each other}? A beautiful way to study this interaction is to see if every element of a ring can be broken down into simpler, well-understood pieces of other elements.

Throughout this paper, \( R \) denotes an associative ring with identity which is \textit{not} necessarily commutative. We employ the standard notations: \( U(R) \) for the group of units, \( Nil(R) \) for the set of nilpotent elements, \( C(R) \) for the center, \( Id(R) \) for the set of idempotents, and \( J(R) \) for the Jacobson radical. For any positive integer \( n \), the ring consisting of all \( n \times n \) matrices over \( R \) is denoted by \( M_{n}(R) \).

Historically, Nicholson introduced the notion of \textit{clean rings} in \cite{nic} as follows: A ring is called \textit{clean} if every element is a sum of a unit and an idempotent. Over the past three decades, these rings have drawn a considerable interest, and their various generalizations and related variants have been emerged.

Later on, Diesl introduced the concept of \textit{nil-clean rings} in \cite{di} thus: A ring is called \textit{nil-clean} if every element can be written as a sum of a nilpotent and an idempotent. He established several fundamental properties and developed an almost comprehensive theory of such rings.

It is well known that the inclusion \( 1 + J(R) \subseteq U(R) \) is true always for any ring \( R \). Thereby, \( R \) is said to be a \textit{JU-ring} provided that the reverse containment also holds, that is, \( U(R) = 1 + J(R) \) (see \cite{D} and \cite{klm} as well). In a similar vein, \( R \) is called a \textit{UU-ring} provided that \( U(R) = 1 + Nil(R) \) (see, e.g., \cite{dl}).

Further on, C\u{a}lug\u{a}reanu and Lam defined in \cite{Clam} a non-zero element \( a \) in a ring \( R \) to be \textit{fine} if \( a \in U(R) + Nil(R) \) calling a ring \( R \) \textit{fine} if each non-zero element is fine, that is, \[ R \setminus \{0\} \subseteq U(R) + Nil(R) .\]

Moreover, Zhou introduced in \cite{z} the concept of \textit{generalized fine rings} like this: A ring \( R \) is said to be \textit{generalized fine} if each element lying outside the Jacobson radical can be written as a sum of a unit and a nilpotent element. That is,
\[
R \setminus J(R) \subseteq U(R) + Nil(R).
\]
Since the opposite inclusion always holds, because of the well-known equality $U(R)+J(R)=U(R)$, this condition is equivalent to
\[
R \setminus J(R) = U(R) + Nil(R).
\]

Furthermore, in \cite{xi}, Xiao and Tong defined the notion of \textit{\(n\)-clean rings} as a common generalization of clean rings: A ring is called \textit{\(n\)-clean} if every element can be expressed as a sum of an idempotent and \(n\) units.

Likewise, in \cite{wang}, the set
\[
\sqrt{J(R)} = \{ x \in R : x^{n} \in J(R) \text{ for some } n \ge 1 \}
\]
was introduced; it properly contains \( J(R) \), but obviously need \textit{not} be a subring of \( R \). Manifestly, the relation \( Nil(R) \subseteq \sqrt{J(R)} \) is fulfilled always.

On the other hand, Saini and Udar recently introduced in \cite{SU} the so-termed \textit{\(\sqrt{J}U\)-rings} as those rings for which \( U(R) = 1 + \sqrt{J(R)} \). One knows that this class properly contains both UU-rings and JU-rings.
Independently, Danchev et al. investigated in \cite{DDE} the same class of rings under the name \textit{\(UJ^\#\)-rings}, and especially their non-trivial extensions, including group rings.

In the present article, we take a step beyond these notions by introducing and exploring a \textbf{new} class of rings called by us \textit{generalized $\sqrt{J}$-fine rings}: We say a ring $R$ is \textit{generalized $\sqrt{J}$-fine} if every element \( a \in R \setminus J(R) \) can be written as a sum of a unit and an element from $\sqrt{J(R)}$. The main difference in our considerations is that instead of using nilpotent elements, we use the larger set $\sqrt{J(R)}$ of elements which genuinely contains \( J(R) \) as well. By allowing this broader collection of elements, we obtain a richer theory that still behaves nicely in many ways.

Our main goal in this paper is to explore the fundamental properties of these rings. We establish that they sit naturally between the classes of generalized fine rings and 2-clean rings (see Example \ref{11}, Theorem \ref{2.14}, and Example \ref{1}). Our most important result is that if a ring is generalized $\sqrt{J}$-fine, then the full matrix ring over it is also generalized $\sqrt{J}$-fine, no matter how large the matrix size is (see Theorem \ref{22}). We also examine when group rings are generalized $\sqrt{J}$-fine, and give a complete answer for the class of locally finite groups (see Theorem \ref{111}).

As a partial answer to an open question due to C\u{a}lug\u{a}reanu and Lam (\cite[Question 5.19]{Clam}), we prove that fine rings are always \(2\)-clean. Moreover, the class of generalized $\sqrt{J}$-fine rings lies strictly between the class of generalized fine rings and the class of \(2\)-clean rings.

Overall the presentation, our work aims to provide the reader with a clear and accessible introduction to generalized $\sqrt{J}$-fine rings by showing how they connect to earlier concepts and why they are worthy of studying in-depth.

\section{Basic Properties of Generalized $\sqrt{J}$-Fine Rings}

We begin our attack with two key instruments.

\begin{definition}\label{maj1}
Let \(R\) be a ring. We say a non-zero element \(r\in R\) is \textit{$\sqrt{J}$-fine} if \(r = u+a\) for some \(a\in\sqrt{J(R)}\) and \(u\in U(R)\); such a decomposition will, hereafter, be known as a \textit{$\sqrt{J}$-fine decomposition}. If $R \neq (0)$ and all elements in $R \setminus \{0\}$ are $\sqrt{J}$-fine, we say that $R$ is a \textit{$\sqrt{J}$-fine} ring. The set of all $\sqrt{J}$-fine elements is designated by $\sqrt{J}F(R)$.
\end{definition}

\begin{definition}\label{maj2}
We say a ring $R$ is \textit{generalized $\sqrt{J}$-fine} if each element of $R \setminus J(R)$ is $\sqrt{J}$-fine.
\end{definition}

The following comments are worthy of recording showing that the newly defined concept of generalized $\sqrt{J}$-fine rings is independent of being generalized fine one.

\begin{example}\label{11}
Every generalized fine ring is a generalized $\sqrt{J}$-fine ring, but the converse is not necessarily true. Consulting with \cite[Example 5]{full}, there exists a domain $R$ such that $R/J(R) \cong M_2(\mathbb{R})$. Hence, $R/J(R)$ is a generalized $\sqrt{J}$-fine ring in conjunction with Lemma \ref{2.4}; therefore, Lemma \ref{2.10} allows us to conclude that $R$ is a generalized $\sqrt{J}$-fine ring. However, since $\operatorname{Nil}(R) = (0)$, $R$ cannot be a generalized fine ring.
\end{example}

Our useful machinery is the next preliminary assertion.

\begin{proposition}\label{2}
For any ring \(R\), the following three points hold:

(1) If \(R \neq 0\), then \(\sqrt{J}F(R) \cap C(R) = U(C(R))\).

(2) \(v \in U(R)\) implies \(v\sqrt{J}F(R)v^{-1} = \sqrt{J}F(R)\). If, moreover, \(v \in U(R) \cap C(R)\), then \(v\sqrt{J}F(R) = \sqrt{J}F(R)\).

(3) If \(R \neq 0\) and \(\sqrt{J(R)} \subseteq J(R)\) (i.e., \(\sqrt{J(R)} = J(R)\)), then \(\sqrt{J}F(R) = U(R)\).
\end{proposition}

\begin{proof}
(1) We know that \(U(C(R))\subseteq \sqrt{J}F(R) \cap C(R)\). Conversely, take \(r \in  \sqrt{J}F(R)\cap C(R)\) and consider a $\sqrt{J}$-fine decomposition \(r = u + a\). Since \(u\) commutes with \(a = r - u\), we have
	\[
	r = u(1 + u^{-1}a) \in U(R) \cdot U(R) \subseteq U(R).
	\]
Thus, \(r \in U(C(R))\), as required.
	
(2) Choose $j\in \sqrt{J(R)}$. It is easy to see that $vjv^{-1}\in \sqrt{J(R)}$. So, the first statement is obvious. For the second one, if \(r = u + a\) is a $\sqrt{J}$-fine decomposition of \(r \in \sqrt{J}F(R)\) and \(v \in U(R) \cap C(R)\), then
	\[
	vr = vu + va
	\]
is a $\sqrt{J}$-fine decomposition of \(va\), whence \(v\sqrt{J}F(R) \subseteq \sqrt{J}F(R)\). Replacing \(v\) by \(v^{-1}\) gives the requested equality.
	
(3) Under the given hypothesis, take any \(r \in R\) with $\sqrt{J}$-fine decomposition \(r = u + a\). Then, \(a \in J(R)=\sqrt{J(R)}\). Consequently,
	\[
	r \in U(R) + J(R) \subseteq U(R),
	\]
and, therefore, \(\sqrt{J}F(R) = U(R)\), as needed.
\end{proof}

Our first principal result sounds as follows.

\begin{theorem}\label{2-good}
For any $\sqrt{J}$-fine ring \(R\), the following two items are valid:

(1) \(R = \{1\} \cup \bigl(U(R) + U(R)\bigr)\).

(2) \(R\) is 2-good if and only if \(|R| \neq 2\).
\end{theorem}

\begin{proof}
(1) Take any \(r \in R\). If \(r = 1\), we are done. Otherwise, \(r - 1 \in \sqrt{J}F(R)\) and thus it admits a $\sqrt{J}$-fine decomposition \(r - 1 = u + a\) with \(u \in U(R)\) and \(a \in \sqrt{J(R)}\). Hence,
\[
r = u + (1 + a) \in U(R) + U(R),
\]
because \(1 + a \in U(R)\). This establishes (1).
	
(2) The ``if'' part is pretty easy. Conversely, suppose \(|R| > 2\) and assume, for contradiction, that \(R\) is 2-good. We first show that \(U(R) \neq \{1\}\). If \(U(R) = \{1\}\), then any non-zero \(r \in R\) has a $\sqrt{J}$-fine decomposition \(r = 1 + a\) with \(a \in \sqrt{J(R)}\); consequently, \(r \in U(R)\) and hence \(r = 1\). This would force \(|R| = 2\), contradicting the hypothesis. So, there exists a unit \(v \neq 1\).
	
Now, consider a $\sqrt{J}$-fine decomposition of \(1 - v\neq0\):
	\[
	1 - v = w + b \quad (w \in U(R),\; b \in \sqrt{J(R)}).
	\]
Set \(x := w + v = 1 - b\). Since $b \in \sqrt{J(R)}$, it must be that \(x \in U(R)\). Therefore,
	\[
	1 = x^{-1}w + x^{-1}v \in U(R) + U(R),
	\]
which means that \(1\) is a sum of two units. Thus, \(R\) is 2-good, as wanted.
\end{proof}

We proceed by proving the following helpful technicality.

\begin{lemma}\label{2.10}
Let \( I \) be an ideal of a ring \( R \) with \( I \subseteq J(R) \). Then, \( R \) is generalized $\sqrt{J}$-fine if and only if \( R/I \) is generalized $\sqrt{J}$-fine.
\end{lemma}

\begin{proof}
Since \( I \subseteq J(R) \), we may write \( J(R/I) = J(R)/I \). Also, activating \cite[Lemma 2.1(5)]{DDE}, we have \(\sqrt{J(R/I)} = \sqrt{J(R)}/I\).
	
(\(\Rightarrow\)) Suppose \( R \) is generalized \(\sqrt{J}\)-fine. Take \( \bar{r} \in R/I \) with \( \bar{r} \notin J(R/I) \). Then, \( r \notin J(R) \). Since \( R \) is generalized \(\sqrt{J}\)-fine, we can write \( r = u + a \) with \( u \in U(R) \) and \( a \in \sqrt{J(R)} \). Hence, \( \bar{r} = \bar{u} + \bar{a} \), where \( \bar{u} \in U(R/I) \) and \( \bar{a} \in \sqrt{J(R/I)} \) (cf. \cite[Lemma 2.1(5)]{DDE}). So, \( R/I \) is generalized \(\sqrt{J}\)-fine.
	
(\(\Leftarrow\)) Now, suppose \( R/I \) is generalized \(\sqrt{J}\)-fine. Take \( r \notin J(R) \). Then, \( \bar{r} \notin J(R/I) \), and so \( \bar{r} = \bar{u} + \bar{a} \) for some \( \bar{u} \in U(R/I) \) and \( \bar{a} \in \sqrt{J(R/I)} \). Since \( I \subseteq J(R) \), we can lift \( \bar{u} \) to \( u \in U(R) \). Also, \cite[Lemma 2.1(5)]{DDE} guarantees that we can lift \( \bar{a} \) to \( a \in \sqrt{J(R)} \). Thus, \( r - u - a \in I \subseteq J(R) \). Letting \( j := r - u - a \in J(R) \), we deduce
$
	r = u + a + j.
$
Since \( a \in \sqrt{J(R)} \) and \( j \in J(R) \), \cite[Lemma 2.11(1)]{DDE} works to get \( a + j \in \sqrt{J(R)} \). Therefore, \( R \) is generalized \(\sqrt{J}\)-fine, as desired.
\end{proof}

As a consequence, we yield.

\begin{corollary}
Let \( R \) be a ring, \( \sigma \) an endomorphism of \( R \), and \( M \) a non-trivial bi-module over \( R \). Then, the following statements are fulfilled:
	
(1) The trivial extension \( R \propto M \) is a generalized \(\sqrt{J}\)-fine ring if and only if \( R \) is a generalized \(\sqrt{J}\)-fine ring.
	
(2) For \( n \ge 2 \), the quotient ring \( R[x]/(x^n) \) is generalized \(\sqrt{J}\)-fine if and only if \( R \) is generalized \(\sqrt{J}\)-fine.

(3) \( R[[x;\sigma]] \) is a generalized \(\sqrt{J}\)-fine ring if and only if \( R \) is a generalized \(\sqrt{J}\)-fine ring.
\end{corollary}

Standardly, the symbol $Nil_*(R)$ stands for the \textit{lower nil-radical} of a ring $R$. Recall
that a ring $R$ is termed \textit{$2$-primal}, provided that $Nil(R)=Nil_*(R)$. For an endomorphism $\alpha$ of $R$, a ring $R$ is said to be $\alpha$-compatible if, for any elements $a$ and $b$ in $R$, the equality $ab = 0$ holds if, and only if, $a\alpha(b) = 0$. This definition is firstly given in \cite{amin}. In this case, the map $\alpha$ must be injective.

\begin{lemma}\label{fi}
Assume that $R$ is a $2$-primal ring and $\alpha$-compatible. Then,
\[
Nil_*(R)[x; \alpha] = J(R[x; \alpha]) = \sqrt{J(R[x; \alpha])}.
\]
\end{lemma}

\begin{proof}
It suffices to show that $Nil_*(R)[x; \alpha] \subseteq J(R[x; \alpha])$, and that $\sqrt{J(R[x; \alpha])} \subseteq Nil_*(R)[x; \alpha]$. Choosing $f=\sum_{i=0}^{n}a_ix^i \in Nil_*(R)[x; \alpha]$, it is then easily seen, for any $g=\sum_{i=0}^{n}b_ix^i \in R[x; \alpha]$, that $fg \in Nil_*(R)[x; \alpha]$, whence from \cite[Lemma 2.1]{che} we perceive that $1-fg \in U(R[x; \alpha])$ implying $f \in J(R[x; \alpha])$.

Now, assuming $f=\sum_{i=0}^{n}a_ix^i \in \sqrt{J(R[x; \alpha])}$, \cite[Lemma 2.1(3)]{DDE} assures that $1-f \in U(R[x; \alpha])$, whence from \cite[Lemma 2.1]{che} we infer for any $1\le i \le n$ that $a_i \in Nil_*(R)$. 

On the other hand, since $f \in \sqrt{J(R[x; \alpha])}$ exists, there is $m \in \mathbb{N}$ with $f^m \in J(R[x; \alpha])$, so that $1-f^m x \in U(R[x; \alpha])$ and then again from \cite[Lemma 2.1]{che} we conclude that $a_0^m \in Nil(R)$. Hence, $a_0 \in Nil(R)$, and since $R$ is a $2$-primal ring, it follows that $a_0 \in Nil_*(R)$. Thus, for every $0 \le i \le n$, we derive $a_i \in Nil_*(R)$, as asked.
\end{proof}

We now define
\[
G\sqrt{J}F(R):=\left\{ a \in R \setminus J(R) \mid \exists u \in U(R),\; j \in \sqrt{J(R)},\text{ s.t. } a=u+j  \right\}.
\]

The following technicality is now true.

\begin{lemma}
Let $R$ be an $\alpha$-compatible ring. Then, the following two conditions are valid:

(1) $R$ is a $2$-primal ring if and only if $$G\sqrt{J}F(R[x; \alpha])=U(R)+\operatorname{Nil}_*(R)[x; \alpha]x.$$

(2) $R$ is a reduced ring if and only if $$G\sqrt{J}F(R[x; \alpha])=U(R).$$
\end{lemma}

\begin{proof}
(1) Assume that $R$ is a $2$-primal ring, and write $f:=\sum_{i=0}^{n}a_ix^i \in G\sqrt{J}F(R[x; \alpha])$. Then, there exist $u=\sum_{i=0}^{n}u_ix^i \in U(R[x; \alpha])$ and $d=\sum_{i=0}^{n}d_ix^i \in \sqrt{J(R[x; \alpha])}$ such that $f=u+d$. From Lemma \ref{fi}, we obtain for each $0\le i\le n$ that $d_i \in \operatorname{Nil}_*(R)$, and also from \cite[Lemma 2.1]{che}, we have that $u_0 \in U(R)$ and, for every $1\le i\le n$, that $u_i \in \operatorname{Nil}_*(R)$. Hence,
\[
f=(u_0+d_0)+\sum_{i=1}^{n}(u_i +d_i)x^i.
\]
Since $\operatorname{Nil}_*(R)$ is a nil-ideal, we find $u_0+d_0 \in U(R)$ and, for each $1\le i\le n$, $u_i+d_i \in \operatorname{Nil}_*(R)$ yielding that $f \in U(R)+\operatorname{Nil}_*(R)[x; \alpha]x$.

Conversely, thanks to \cite[Lemma 2.1]{che}, we have $$U(R)+\operatorname{Nil}_*(R)[x; \alpha]x \subseteq U(R[x; \alpha]),$$ which ensures that $$U(R)+\operatorname{Nil}_*(R)[x; \alpha]x \subseteq G\sqrt{J}F(R[x; \alpha]).$$

Now reciprocally, assume that
\[
G\sqrt{J}F(R[x; \alpha])=U(R)+\operatorname{Nil}_*(R)[x; \alpha]x
\]
holds, and let $a \in \operatorname{Nil}(R)$. Then, there exists $n \in \mathbb{N}$ such that $a^n=0$, so from \cite[Lemma 3.1]{chen} we detect
\[
(ax)^n=a\alpha(a)\cdots \alpha^{n-1}(a)x^n=0.
\]
Thus, $ax \in \sqrt{J(R[x; \alpha])}$, and since $1+ax \notin \sqrt{J(R[x; \alpha])}$ it must be that $1+ax \in G\sqrt{J}F(R[x; \alpha])$ forcing that $a \in \operatorname{Nil}_*(R)$. Finally, $R$ is a $2$-primal ring, as pursued.

(2) Since $R$ is reduced, we know that $\operatorname{Nil}(R)=(0)$. Then, similarly to part (1), the result follows at once.
\end{proof}

We further continue with a series of technical claims.

\begin{proposition}\label{2.13}
Let \( R \) be a generalized \(\sqrt{J}\)-fine ring. Then, the following two issues are fulfilled:

(1) \( R = (1 + J(R)) \cup (U(R) + U(R)) \).

(2) \( R \) is 2-good if and only if \( R = J(R) \cup (1 + J(R)) \).
\end{proposition}

\begin{proof}
(1) Take any \( r \in R \setminus (1 + J(R)) \). Then, \( 1 - r \not\in J(R) \). Since \( R \) is generalized \(\sqrt{J}\)-fine, we can write \( 1 - r = u + a \) with \( u \in U(R) \) and \( a \in \sqrt{J(R)} \). Hence,
	\[
	r = (-u) + (1 - a).
	\]
Now, \( -u \in U(R) \) and \( 1 - a \in U(R) \). So, \( r \in U(R) + U(R) \). Thus, \( R = (1 + J(R)) \cup (U(R) + U(R)) \).
	
(2) (\(\Rightarrow\)). Suppose \( R \) is 2-good. Then, \( R/J(R) \) is also 2-good, so \( |R/J(R)| > 2 \). Therefore, \( R = J(R) \cup (1 + J(R)) \).
	
(\(\Leftarrow\)). Assume \( R = J(R) \cup (1 + J(R)) \). We first intend to show that \( U(R) \neq 1 + J(R) \). To that target, suppose that \( U(R) = 1 + J(R) \). If \( r \in R \setminus J(R) \), then since \( R \) is generalized \(\sqrt{J}\)-fine, we can write \( r = v + a \) with \( v \in U(R) \) and \( a \in \sqrt{J(R)} \). But, we are assuming \( U(R) = 1 + J(R) \), so that \( v = 1 + j \) for some \( j \in J(R) \). Consequently,
\[
r = 1 + a + j\in U(R)+J(R)=U(R).
\]
Thus, \( r \in 1 + J(R) \), and so \( R = J(R) \cup (1 + J(R)) \). This, however, contradicts the assumption that \( U(R) \neq 1 + J(R) \). Therefore, we must have \( U(R) = 1 + J(R) \).

Now, take any \( u \in U(R) \setminus (1 + J(R)) \). Then, \( u - 1 \not\in J(R) \). Since \( R \) is generalized \(\sqrt{J}\)-fine, we can write \( u - 1 = v + a \) with \( v \in U(R) \) and \( a \in \sqrt{J(R)} \). Hence,
$
u = v + (1 + a).
$
Here \( v \in U(R) \) and \( 1 +a \in U(R) \). So, \( u \in U(R) + U(R) \). Therefore, every element of \( U(R) \setminus (1 + J(R)) \) lies in \( U(R) + U(R) \). In particular, \( 1 + J(R) \subseteq U(R) + U(R) \). Combining this with part (1), we get
$
R = U(R) + U(R),
$
as claimed.
\end{proof}

\begin{proposition}\label{2.5}
If \( R \) is a generalized \(\sqrt{J}\)-fine ring and \( \sqrt{J(R)} \subseteq J(R) \), then it is generalized fine.
\end{proposition}

\begin{proof}
Assume \( \sqrt{J(R)} \subseteq J(R) \), and take \( r \in R \setminus J(R) \). Since \( R \) is generalized \(\sqrt{J}\)-fine, we can write \( r = u + a \) for some \( u \in U(R) \) and \( a \in \sqrt{J(R)} \). Then, \( a \in J(R) \), so that \( r \in U(R) \). Thus, $R$ is local and, therefore, \( R \) is generalized fine, as asserted.
\end{proof}

\begin{proposition}\label{2.15}
The following three statements are equivalent for any ring \(R\):

(1) \(R\) is generalized \(\sqrt{J}\)-fine and \(JU\).

(2) \(R\) is generalized \(\sqrt{J}\)-fine and \(\sqrt{J}U\).

(3) \(R/J(R) \cong \mathbb{F}_2\).
\end{proposition}

\begin{proof}
Both (3) \(\Rightarrow\) (1) and (3) \(\Rightarrow\) (2) are straightforward.
	
(1) \(\Rightarrow\) (3). Assume \(R\) is both generalized \(\sqrt{J}\)-fine and \(JU\). Take any \(r \in R \setminus J(R)\). Then, we can write \(r = u + a\) with \(u \in U(R)\) and \(a \in \sqrt{J(R)}\). Since \(R\) is \(UJ\), we have \(U(R) = 1 + J(R)\), so \(u = 1 + j\) for some \(j \in J(R)\). Hence,
	\[
	r = 1 + a + j.
	\]
That is why, $r\in U(R)+J(R)=U(R)$. Thus, \(r \in 1 + J(R)\). Therefore, \(R = J(R) \cup (1 + J(R))\), so that \(R/J(R) \cong \mathbb{F}_2\).
	
(2) \(\Rightarrow\) (3). Assume \(R\) is both generalized \(\sqrt{J}\)-fine and \(U\sqrt{J}\). Then, Lemma \ref{2.8} enables us that \(R/J(R)\) is \(\sqrt{J}\)-fine and, moreover, it is \(U\sqrt{J}\) invoking \cite[Theoem 2.10]{DDE}. Set $R/J(R):=S$. Suppose, for contradiction, that \(|S| > 2\). In the proof of Theorem~\ref{2-good}(2) we established that, under these circumstances, one can write \(1 = u_1 + u_2\) for some units \(u_1, u_2 \in U(S)\). Since \(S\) is $U\sqrt{J}$, we have \(u_1 \in 1 + \sqrt{J(S)}\) or, equivalently, \(u_1 - 1 \in \sqrt{J(S)}\). But, \(u_1 - 1 = -u_2\) is also a unit. Thus, $u_1-1\in U(S)\cap \sqrt{J(S)}$ leading to a contradiction in view of \cite[Proposition 2.3]{DDE}. Therefore, \(|S| = 2\) and hence \(S \cong \mathbb{F}_2\), as stated.
\end{proof}

A valuable consequence is the following one.

\begin{corollary}\label{2.6}
Let \( R \) be a generalized \(\sqrt{J}\)-fine ring. Then, the following two assertions hold:
	
(1) If \( R \) is a UU-ring, then it is generalized fine.
	
(2) If \( R \) is a JU-ring, then it is generalized fine.
\end{corollary}

\begin{proof}
(1) Assume \( R \) is simultaneously generalized \(\sqrt{J}\)-fine and UU. Take any \( a \in \sqrt{J(R)} \). Then, \( 1 + a \in U(R) = 1 + Nil(R) \), so \( a \in Nil(R) \). Hence, \( \sqrt{J(R)} \subseteq \mathrm{Nil}(R) \), and in virtue of Proposition \ref{2.5} the ring \( R \) is generalized fine.
	
(2) Now, assume \( R \) is simultaneously generalized \(\sqrt{J}\)-fine and JU. For any \( a \in \sqrt{J(R)} \), we have \( 1 + a \in U(R) = 1 + J(R) \), so \( a \in J(R) \). Thus, \( \sqrt{J(R)} = J(R) \), and again Proposition \ref{2.5} applies to get that \( R \) is generalized fine, as we want.
\end{proof}

\begin{proposition}\label{2.20}
Let \( R \) be a generalized \(\sqrt{J}\)-fine ring and \( e \in Id(R) \cap C(R) \). Then, the corner subring \( eRe \) is also generalized \(\sqrt{J}\)-fine.
\end{proposition}

\begin{proof}
Take any \( r \in eRe \setminus J(eRe) \). Then, \( r \not\in J(R) \). Since \( R \) is generalized \(\sqrt{J}\)-fine, we can write \( r = u + a \) for some \( u \in U(R) \) and \( a \in \sqrt{J(R)} \). But,
\[
	r = ere = eue + eae,
\]
and using \cite[Lemma 2.8]{DDE} we extract \( eae \in \sqrt{J(eRe)} \). Also, because \( e \) is central, \( eue \in U(eRe) \). Finally, \( eRe \) is generalized \(\sqrt{J}\)-fine, as we desire.
\end{proof}

Recollect that a ring $R$ is \textit{abelian} whenever $Id(R)\subseteq C(R)$.

\begin{proposition}\label{2.25}
Every abelian generalized \(\sqrt{J}\)-fine ring is indecomposable.
\end{proposition}

\begin{proof}
Given \( 0 \neq e \in R \) is a central idempotent. Then, \( e \not\in J(R) \), and so \( e = u + a \) for some \( u \in U(R) \) and \( a \in \sqrt{J(R)} \). Hence, \( e = u(1 + u^{-1}a) \). But, since \( e \) is central, the elements \( u \) and \( a \) commute, so that \( u^{-1}a \in \sqrt{J(R)} \) viewing \cite[Lemma2.1(1)]{DDE}. Thus, \( 1 + u^{-1}a \in U(R) \) employing \cite[Lemma2.1(3)]{DDE}, whence \( e \in U(R) \) and so \( e = 1 \). Therefore, \( R \) has no non-trivial central idempotents, and that is why it is indecomposable, as promised.
\end{proof}

We now come to the following extra machinery.

\begin{definition}\label{2.32}
A ring \( R \) is called \textit{strongly generalized \(\sqrt{J}\)-fine} if every element \( r \in R \setminus J(R) \) can be written as \( r = u + a \), where \( u \in U(R) \), \( a\in \sqrt{J(R)} \), and \( au = ua \). A ring \( R \) is called \textit{uniquely generalized \(\sqrt{J}\)-fine} if every element \( r \in R \setminus J(R) \) has a unique decomposition \( r = u + a \) with \( u \in U(R) \) and \( a \in \sqrt{J(R)} \).
\end{definition}

Our second main achievement states the following.

\begin{theorem}\label{2.33}
For an arbitrary ring \( R \), the following two claims are true:
	
(1) \( R \) is strongly generalized \(\sqrt{J}\)-fine if and only if \( R \) is local.
	
(2) \( R \) is uniquely generalized \(\sqrt{J}\)-fine if and only if \( R \) is a division ring.
\end{theorem}

\begin{proof}
(1) Suppose \( R \) is strongly generalized \(\sqrt{J}\)-fine. Take \( r \in R \setminus J(R) \). Then, \( r = u + a \) with \( u \in U(R) \), \( a \in \sqrt{J(R)} \), and \( ua = au \). Hence, one derives that \( ru^{-1} = 1 + au^{-1} \in U(R) \) in accordance with \cite[Lemma 2.1(1), (3)]{DDE}. So, \( r \in U(R) \). Thus, \( R = J(R) \cup U(R) \), so that \( R \) is local, indeed. 

Conversely, if \( R \) is local, then each element outside \( J(R) \) is a unit, so we may just write \( r = r + 0 \). Hence, \( R \) is strongly generalized \(\sqrt{J}\)-fine, as formulated.
	
(2) Suppose \( R \) is uniquely generalized \(\sqrt{J}\)-fine. First, we manage to show that \( \sqrt{J(R)} = (0) \). To that end, take \( a \in \sqrt{J(R)} \). Then, \( 1 + a \in U(R) \), and so \( 1 + a = (1 + a) + 0 \) has two decompositions. By uniqueness, \( a = 0 \). Hence, \( \sqrt{J(R)} = (0) \) giving automatically \( J(R) = (0) \). Now, choose any non-zero \( r \in R \). Thus, \( r \not\in J(R) \), so that \( r = u + a \) uniquely. Since \( a = 0 \), we arrive at \( r = u \in U(R) \). Therefore, any non-zero element is a unit, so \( R \) is indeed a division ring. 

Conversely, if \( R \) is a division ring, then every \( a \in R \setminus J(R) \) has the unique decomposition \( a = a + 0 \). Hence, \( R \) is uniquely generalized \(\sqrt{J}\)-fine, as formulated.
\end{proof}

The next commentaries are worthwhile.

\begin{example}\label{direct}
(1) The direct product of generalized \( \sqrt{J} \)-fine rings is \textit{not} necessarily generalized \( \sqrt{J} \)-fine. For example, \( \mathbb{Z}_2 \) is generalized \( \sqrt{J} \)-fine, but \( \mathbb{Z}_2 \times \mathbb{Z}_2 \) is obviously not.

(2) The upper triangular matrix ring \( T_2(\mathbb{Z}_2) \) is \textit{not} a generalized \(\sqrt{J}\)-fine ring. To see this, put
$
I := \{(a_{ij}) \in T_2(\mathbb{Z}_2) \mid a_{ii} = 0\}.
$
Then, one verifies that \[ I \subseteq J(T_2(\mathbb{Z}_2)) ~ {\rm and} ~  T_2(\mathbb{Z}_2)/I \cong \mathbb{Z}_2 \times \mathbb{Z}_2 .\] The result now follows immediately from a combination of (1) and Lemma \ref{2.10}.

(3) A subring of a generalized \(\sqrt{J}\)-fine ring need \textit{not} be generalized \(\sqrt{J}\)-fine. For instance, the matrix ring \(M_2(\mathbb{Z}_2)\) is generalized $\sqrt{J}$-fine owing to Lemma~\ref{2.4} alluded to below. But, as we checked above in point (2), the triangular matrix ring \(T_2(\mathbb{Z}_2)\) is not generalized $\sqrt{J}$-fine, despite being a subring of \(M_2(\mathbb{Z}_2)\).
\end{example}

\begin{theorem}\label{2.14}
Every generalized \(\sqrt{J}\)-fine ring is \(2\)-clean.
\end{theorem}

\begin{proof}
Let \( R \) be a generalized \(\sqrt{J}\)-fine ring. If \( r \in R \setminus J(R) \), then \( r = u + a \) for some \( u \in U(R) \) and \( a \in \sqrt{J(R)} \). Since \( a - 1 \) is a unit, we deduce \( r = u + (a - 1) + 1 \), which apparently is a \( 2 \)-clean decomposition of \( r \), as required. 

If, however, \( r \in J(R) \), then \( r = 1 + (r - 1) + 0 \), where \( 1, r - 1 \in U(R) \) and \( 0 \) is the trivial zero idempotent. Thus, every element of \( R \) is \( 2 \)-clean, and hence \( R \) is a \( 2 \)-clean ring.
\end{proof}

As a direct consequence, we have:

\begin{corollary}
Every fine ring is \(2\)-clean.
\end{corollary}

We finish this section with the following construction.

\begin{example}\label{1}
Take \( R = \mathbb{Z}_2 \times \mathbb{Z}_2 \) and \( S = T_2(\mathbb{Z}_2) \). Both are \( 2 \)-clean, but a simple inspection shows that neither of them is generalized \( \sqrt{J} \)-fine. So, we may list the following proper inclusions:

\[
\{\text{fine rings}\} \subsetneqq \{\text{generalized fine}\} \subsetneqq \{\text{generalized $\sqrt{J}$-fine}\} \subsetneqq \{\text{2-clean rings}\}
\]
\end{example}

\section{Matrix and Group Rings}

In this section, we establish some crucial results about the structure of matrix and group rings under the new point of view as stated in Definitions~\ref{maj1} and \ref{maj2}, respectively. 

Our work here starts with the following pivotal statement.

\begin{proposition}\label{lemma2}
Let $R$ be a ring, and let $X =(a_{ij})$ be the upper triangular matrix in ${\rm M}_n(R)$. Then, $X \in \sqrt{J({\rm M}_n(R))}$ if and only if $a_{ii} \in \sqrt{J(R)}$ for all $1 \le i \le n$.
\end{proposition}

\begin{proof}
Suppose $X \in \sqrt{J({\rm M}_n(R))}$. Then, there exists $s \in \mathbb{N}$ such that $$X^s \in J({\rm M}_n(R))={\rm M}_n(J(R)).$$ This insures that $a^s_{ii} \in J(R)$ for all $1 \le i \le n$, whence $a_{ii} \in \sqrt{J(R)}$ for all $1 \le i \le n$.
	
If, however, $a_{ii} \in \sqrt{J(R)}$ for all $1 \le i \le n$, then, for each $1 \le i \le n$, there is $s_i \in \mathbb{N}$ such that $a^{s_i}_{ii} \in J(R)$. Setting $t:=\max\{s_i\mid 1 \le i \le n\}$, we observe that $a^t_{ii} \in J(R)$ for all $1 \le i \le n$. Now, we calculate
\[X^t = \begin{pNiceArray}{cccc}
		a^t_{11} &  & \Block{2-2}<\Large>{\mathbf{\ast}} \\
		& a^t_{22} \\
		\Block{2-2}<\Large>{0} && \ddots &  \\
		&&  & a^t_{nn}
	\end{pNiceArray}
	= \begin{pNiceArray}{cccc}
		a^t_{11} &  & \Block{2-2}<\Large>{0} \\
		& a^t_{22} \\
		\Block{2-2}<\Large>{0} && \ddots &  \\
		&&  & a^t_{nn}
	\end{pNiceArray}
	+ \begin{pNiceArray}{cccc}
		0 &  & \Block{2-2}<\Large>{\mathbf{\ast}} \\
		& 0 \\
		\Block{2-2}<\Large>{0} && \ddots &  \\
		&&  & 0
	\end{pNiceArray}.
\]
We now inspect that the first right hand side matrix belongs to $J({\rm M}_n(R))$, while the second one is nilpotent. Consequently, applying \cite[Lemma 2.1(8)]{DDE}, we deduce that $X^t \in \sqrt{J({\rm M}_n(R))}$, and hence $X \in \sqrt{J({\rm M}_n(R))}$, as we need.
\end{proof}

\begin{proposition}\label{lemma4}
Let $R$ be a ring, and let $A = \begin{pmatrix} X & C \\ 0 & Y \end{pmatrix}$ be a block matrix, where $X \in \sqrt{J({\rm M}_s(R))}$ and $Y \in \sqrt{J({\rm M}_t(R))}$. Then, $A \in \sqrt{J({\rm M}_{s+t}(R))}$.
\end{proposition}

\begin{proof}
Since $X \in \sqrt{J({\rm M}_s(R))}$ and $Y \in \sqrt{J({\rm M}_t(R))}$, there are $k_1,k_2 \in \mathbb{N}$ such that $X^{k_1} \in J({\rm M}_s(R))$ and $Y^{k_2} \in J({\rm M}_t(R))$. Set $k := \max\{k_1,k_2\}$. Then, one computes that
\[
	A^k = \begin{pmatrix} X^k & 0 \\ 0 & Y^k \end{pmatrix} + \begin{pmatrix} 0 & * \\ 0 & 0 \end{pmatrix} \in J({\rm M}_{s+t}(R)) + Nil({\rm M}_{s+t}(R)).
\]
Therefore, \cite[Lemma 2.1(8)]{DDE} leads to $A \in \sqrt{J({\rm M}_{s+t}(R))}$, as desired.
\end{proof}

\begin{proposition}\label{block}
Let
$
	M = \begin{pmatrix} X & C \\ D & Y \end{pmatrix} \in R = M_n(S),
$
where \(X \in M_k(S)\) and \(Y \in M_{n-k}(S)\) with \(1 \le k \le n-1\).
If \(X \in \sqrt{J}F(M_k(S))\) and \(Y \in \sqrt{J}F(M_{n-k}(S))\), then \(M \in \sqrt{J}F(R)\).
\end{proposition}

\begin{proof}
Since \(X\) and \(Y\) are $\sqrt{J}$-fine, we can write
	\[
	X = U + A \quad \text{and} \quad Y = U' + A'
	\]
for some \(U \in U(M_k(S))\), \(U' \in U(M_{n-k}(S))\) and \(A \in \sqrt{J(M_k(S))}\), \(A' \in \sqrt{J(M_{n-k}(S))}\).
	
Now, consider the decomposition

\[
M = \begin{pmatrix} U & 0 \\ D & U' \end{pmatrix} + \begin{pmatrix} A & C \\ 0 & A' \end{pmatrix}.
\qquad\qquad\qquad\qquad\qquad (*)
\]
	
As the first matrix on the right-hand side is block triangular with invertible diagonal blocks, it is invertible in \(R\). The second matrix is in $\sqrt{J(M_n(S))}$ thanking to Proposition \ref{lemma4}. Thus, $(*)$ is really a $\sqrt{J}$-fine decomposition of \(M\) in \(R\), proving that \(M \in \sqrt{J}F(R)\), as it should be.
\end{proof}

An application of Proposition \ref{block} inductively on the number of diagonal blocks is a guarantor of validity of the following result.

\begin{corollary}\label{Block}
Let \( M = (A_{ij}) \in R = M_n(S) \) be a block matrix with each diagonal block \( A_{ii} \) \(\sqrt{J}\)-fine. Then, \( M \) belongs to \(\sqrt{J}F(R)\).
\end{corollary}

The next assertion is critical for the successful establishment of the third major result quoted in the sequel.

\begin{lemma}\label{2.4}
If \( R \) is a generalized $\sqrt{J}$-fine ring, then \( M_2(R) \) is also a generalized $\sqrt{J}$-fine ring.
\end{lemma}

\begin{proof}
Let
	\[
	A := \begin{pmatrix} a & b \\ c & d \end{pmatrix} \in M_2(R) \setminus M_2(J(R)).
	\]
We consider four possible cases based on how many entries of \( A \) are outside \( J(R) \).

\medskip
	
\textbf{Case 1: All four entries are not in \( J(R) \).} Then, both \( a \) and \( d \) are $\sqrt{J}$-fine in \( R \). By Corollary \ref{Block}, \( A \) is $\sqrt{J}$-fine.

\medskip
	
\textbf{Case 2: Exactly three entries are not in \( J(R) \).} So, exactly one entry is in \( J(R) \).

\medskip
	
If \( a \in J(R) \), then \( A \) is $\sqrt{J}$-fine if and only if \( \begin{pmatrix} 0 & b \\ c & d \end{pmatrix} \) is $\sqrt{J}$-fine. So, we may assume \( a = 0 \).
Write
	\[
	A_1 = \begin{pmatrix} 0 & 1 \\ 1 & 1 \end{pmatrix} \begin{pmatrix} 0 & b \\ c & d \end{pmatrix} \begin{pmatrix} -1 & 1 \\ 1 & 0 \end{pmatrix}
	= \begin{pmatrix} -c+d & c \\ b-c+d & c \end{pmatrix}.
	\]
If \( -c+d \not\in J(R) \), then \( A_1 \) and hence \( A \) is $\sqrt{J}$-fine by virtue of Corollary \ref{Block} and Proposition \ref{2}(2). So, assume \( -c+d \in J(R) \). Then, \( A \) is $\sqrt{J}$-fine if and only if \( A_1 \) is $\sqrt{J}$-fine, which amounts to \( \begin{pmatrix} 0 & c \\ b & c \end{pmatrix} \) being $\sqrt{J}$-fine. So, we may assume
$
	A_1 = \begin{pmatrix} 0 & c \\ b & c \end{pmatrix}.
$
Now, let
	\[
	A_2 := \begin{pmatrix} 1 & 1 \\ 0 & 1 \end{pmatrix} \begin{pmatrix} 0 & c \\ b & c \end{pmatrix} \begin{pmatrix} 1 & -1 \\ 0 & 1 \end{pmatrix}
	= \begin{pmatrix} b & -b+2c \\ b & -b+c \end{pmatrix}.
	\]
If \( -b+c \not\in J(R) \), then \( A_2 \) and hence \( A_1 \) is $\sqrt{J}$-fine by usage of Corollary \ref{Block} and Proposition \ref{2}(2). So, assume \( -b+c \in J(R) \), i.e., \( b = c + j \) for some \( j \in J(R) \). Thus, \( A_1 \) is $\sqrt{J}$-fine if and only if \( \begin{pmatrix} 0 & c \\ c & c \end{pmatrix} \) is $\sqrt{J}$-fine. Since \( c \not\in J(R) \), we can write \( c = u + x \) with \( u \in U(R) \) and \( x \in \sqrt{J(R)} \). Therefore,
	\[
	\begin{pmatrix} 0 & c \\ c & c \end{pmatrix}
	= \begin{pmatrix} 0 & u \\ u & c \end{pmatrix} + \begin{pmatrix} 0 & x \\ x & 0 \end{pmatrix},
	\]
which is a unit plus an element from $\sqrt{J}$. So, \( A \) is $\sqrt{J}$-fine.
	
If \( b \in J(R) \), then \( a \) and \( d \) are $\sqrt{J}$-fine in \( R \), so \( A \) is $\sqrt{J}$-fine by Corollary \ref{Block}.
	
The cases \( c \in J(R) \) or \( d \in J(R) \) are quite similar, so we skip their evidences.

\medskip
	
\textbf{Case 3: Exactly two entries are not in \( J(R) \).}

\medskip
	
If \( a \not\in J(R) \) and \( d \not\in J(R) \), then both \( a \) and \( d \) are $\sqrt{J}$-fine, so \( A \) is $\sqrt{J}$-fine by Corollary \ref{Block}.
	
If \( a \not\in J(R) \) and \( b \not\in J(R) \), then \( b \) is $\sqrt{J}$-fine in \( R \), and \( A \) is $\sqrt{J}$-fine if and only if \( \begin{pmatrix} a & b \\ 0 & 0 \end{pmatrix} \) is $\sqrt{J}$-fine. Write \( b = u + x \) with \( u \in U(R) \) and \( x \in \sqrt{J(R)} \). Thus,
	\[
	\begin{pmatrix} a & b \\ 0 & 0 \end{pmatrix}
	= \begin{pmatrix} a & u \\ -1 & 0 \end{pmatrix} + \begin{pmatrix} 0 & x \\ 1 & 0 \end{pmatrix},
	\]
is a sum of a unit and an element from $\sqrt{J}$. So, \( A \) is $\sqrt{J}$-fine.
	
If \( b \not\in J(R) \) and \( c \not\in J(R) \), then \( A \) is $\sqrt{J}$-fine if and only if \( \begin{pmatrix} 0 & b \\ c & 0 \end{pmatrix} \) is $\sqrt{J}$-fine. But,
	\[
	\begin{pmatrix} 0 & 1 \\ 1 & 1 \end{pmatrix} \begin{pmatrix} 0 & b \\ c & 0 \end{pmatrix} \begin{pmatrix} -1 & 1 \\ 1 & 0 \end{pmatrix}
	= \begin{pmatrix} -c & c \\ b-c & c \end{pmatrix},
	\]
which is $\sqrt{J}$-fine by Corollary \ref{Block}. Hence, \( A \) is $\sqrt{J}$-fine with Proposition \ref{2}(2) at hand.
	
The remaining subcases (\( a,c \not\in J(R) \), or \( b,d \not\in J(R) \), or \( c,d \not\in J(R) \)) are rather analogous, and thus we omit their verifications.

\medskip
	
\textbf{Case 4: Exactly one entry is not in \( J(R) \).}

\medskip
	
If \( a \not\in J(R) \), then \( A \) is $\sqrt{J}$-fine if and only if \( \begin{pmatrix} a & 0 \\ 0 & 0 \end{pmatrix} \) is $\sqrt{J}$-fine. But,
	\[
	\begin{pmatrix} 1 & 0 \\ 1 & 1 \end{pmatrix} \begin{pmatrix} a & 0 \\ 0 & 0 \end{pmatrix} \begin{pmatrix} 1 & 0 \\ -1 & 1 \end{pmatrix}
	= \begin{pmatrix} a & 0 \\ a & 0 \end{pmatrix},
	\]
which is $\sqrt{J}$-fine with the help of Case 2. So, \( A \) is $\sqrt{J}$-fine.
	
If \( b \not\in J(R) \), then \( b \) is $\sqrt{J}$-fine, so write \( b = u + x \) with \( u \in U(R) \) and \( x \in \sqrt{J(R)} \). Then, \( A \) is $\sqrt{J}$-fine if and only if \( \begin{pmatrix} 0 & b \\ 0 & 0 \end{pmatrix} \) is $\sqrt{J}$-fine. But,
	\[
	\begin{pmatrix} 0 & b \\ 0 & 0 \end{pmatrix}
	= \begin{pmatrix} 0 & u \\ -1 & 0 \end{pmatrix} + \begin{pmatrix} 0 & x \\ 1 & 0 \end{pmatrix},
	\]
is a sum of a unit and an element from $\sqrt{J}$. So, \( A \) is $\sqrt{J}$-fine.
	
Finally, in all cases, \( A \) is $\sqrt{J}$-fine, and consequently \( M_2(R) \) is generalized $\sqrt{J}$-fine.
\end{proof}

The complete set of \( n \times n \) matrix units is designed by \(\{E_{ij} : 1 \leq i, j \leq n\}\). For any \(i, j\) with \(1 \leq i, j \leq n\) and any \(a \in R\), we define
\[
P_{ij} := I - E_{ii} - E_{jj} + E_{ij} + E_{ji}
\]
and
\[
T_{ij}(a) := I + aE_{ij}.
\]
Note that \(P_{ij}^2 = I\) and \(T_{ij}(a)T_{ij}(-a) = I\) in \(\mathbb{M}_n(R)\) (see \cite{z}).

\medskip

We are now prepared to prove the following third major achievement, as noticed above.

\begin{theorem}\label{22}
If \(R\) is a generalized \(\sqrt{J}\)-fine ring, then \(\mathbb{M}_n(R)\) is also generalized \(\sqrt{J}\)-fine for every positive integer \(n\).
\end{theorem}

\begin{proof}
We induct on $n$. The cases $n=1,2$ are immediate taking into account Lemma \ref{2.4}), so we assume that $n\ge 3$ and that the claim is true for these smaller sizes. Choose $A=(a_{ij})\in \mathbb{M}_n(R)\setminus \mathbb{M}_n(J(R))$ and use the elementary matrices $P_{ij}$ and $T_{ij}(x)$ same as in \cite[Theorem 2.5]{z} noting that the conjugation by them preserves $\sqrt{J}$-fineness.
	
\medskip

\noindent\textbf{Case 1.} $a_{ii},a_{jj}\notin J(R)$ for some $i<j$.

\medskip
	
Write $A=\begin{pmatrix}A_{11}&A_{12}\\A_{21}&A_{22}\end{pmatrix}$. By the induction hypothesis, $A_{11}$ and $A_{22}$ are $\sqrt{J}$-fine. Hence, $A$ is $\sqrt{J}$-fine bearing in mind Corollary \ref{Block}.
	
\medskip

\noindent\textbf{Case 2.} $a_{ii}\in J(R)$ for all $i$.

\medskip
	
We may assume $a_{ii}=0$. Since $A\not\in \mathbb{M}_n(J(R))$, some off-diagonal entry lies outside $J(R)$; thus, conjugating by $P$'s as in \cite[Theorem 2.5]{z}, we reduce to $a_{12}\not\in J(R)$. Then, one obtains that
	\[
	T_{21}(1)AT_{21}(-1)=
	\begin{pmatrix}
		\begin{pmatrix}-a_{12}&a_{12}\\a_{21}-a_{12}&a_{12}\end{pmatrix} & \ast\\
		\ast & \ast
	\end{pmatrix},
	\]
whose leading block has both diagonal entries outside $J(R)$. Therefore, it is $\sqrt{J}$-fine by Case 1, and hence so is $A$.
	
\medskip

\noindent\textbf{Case 3.} Exactly one diagonal entry, say $a_{kk}$, is outside $J(R)$.

\medskip
	
Conjugating by $P_{1k}$, assume $a_{11}\not\in J(R)$. A consultation with Corollary \ref{Block} allows us to suppose $a_{ij}\in J(R)$ for all $i,j\ge 2$, and further $a_{ij}=0$ for $i,j\ge 2$.
	
If, for a moment, $a_{1k}\not\in J(R)$ for some $k>1$, conjugate by $P_{2k}$ to get $a_{12}\not\in J(R)$. For the matrix $P:=\begin{pmatrix}\begin{pmatrix}1&1\\1&0\end{pmatrix}&0\\0&I_{n-2}\end{pmatrix}$, we calculate
	\[
	PAP^{-1}=
	\begin{pmatrix}
		a_{12} & \begin{pmatrix}a_{11}+a_{21}-a_{12}&a_{13}&\cdots&a_{1n}\end{pmatrix}\\
		\begin{pmatrix}a_{12}\\0\\ \vdots\\0\end{pmatrix} &
		\begin{pmatrix}
			a_{11}-a_{12}&a_{13}&\cdots&a_{1n}\\
			a_{31}&0&\cdots&0\\
			\vdots&\vdots&\ddots&\vdots\\
			a_{n1}&0&\cdots&0
		\end{pmatrix}
	\end{pmatrix}.
	\]
Since $a_{12}$ is $\sqrt{J}$-fine, Corollary \ref{Block} allows us to assume the lower-right block is in $\mathbb{M}_{n-1}(J(R))$. Thus,
	\[
	A=
	\begin{pmatrix}
		a_{11}&a_{11}&0&\cdots&0\\
		a_{21}&0&0&\cdots&0\\
		0&0&0&\cdots&0\\
		\vdots&\vdots&\vdots&\ddots&\vdots\\
		0&0&0&\cdots&0
	\end{pmatrix}.
	\]
Then, $T_{31}(1)AT_{31}(-1)$ is $\sqrt{J}$-fine acting with Corollary \ref{Block}; hence, $A$ is $\sqrt{J}$-fine. Note that the case $a_{k1}\not\in J(R)$ is symmetric.
	
If, however, $a_{1i},a_{i1}\in J(R)$ for all $i\ge 2$, then $A-\begin{pmatrix}a_{11}&0\\0&0\end{pmatrix}\in J(\mathbb{M}_n(R))$, so we may write $A=\begin{pmatrix}a_{11}&0\\0&0\end{pmatrix}$. Thus,
	\[
	T_{12}(-1)AT_{12}(1)=
	\begin{pmatrix}
		a_{11}&a_{11}&0&\cdots&0\\
		0&0&0&\cdots&0\\
		\vdots&\vdots&\vdots&\ddots&\vdots\\
		0&0&0&\cdots&0
	\end{pmatrix},
	\]
which is $\sqrt{J}$-fine by the previous paragraph. That is why, $A$ is $\sqrt{J}$-fine.
	
Thereby, all cases yield $A\in\sqrt{J}F(\mathbb{M}_n(R))$. Therefore, $\mathbb{M}_n(R)$ is a generalized $\sqrt{J}$-fine ring, completing the induction and so the entire argumentation.
\end{proof}

\begin{remark}
In general, as already noticed above, triangular matrix rings are \textit{not} necessarily generalized \(\sqrt{J}\)-fine. For instance, a direct check shows that \(T_2(\mathbb{Z}_2)\) clearly does \textit{not} satisfy this property.
\end{remark}

We are now in a position to prove the following.

\begin{lemma}\label{2.8}
Let \( R \) be a generalized $\sqrt{J}$-fine ring. Then, \( R/J(R) \) is $\sqrt{J}$-fine and \( C(R) \) is local.
\end{lemma}

\begin{proof}
Take any non-zero \( \bar{r} \in \bar{R} = R/J(R) \). Then, \( r \not\in J(R) \), so the element \( r \) is generalized \( \sqrt{J} \)-fine in \( R \). Hence, \( r = u + a \) for some \( u \in U(R) \) and \( a \in \sqrt{J(R)} \). Evidently, \( \bar{u} \in U(\bar{R}) \), and via \cite[Lemma 2.1(5)]{DDE}, we have \( \bar{a} \in \sqrt{J(\bar{R})} \). Therefore, \( \bar{r} \) is generalized \( \sqrt{J} \)-fine in \( \bar{R} \).
	
Now, let \( C = C(R) \). An element of \( C \) is a unit in \( C \) if and only if it is a unit in \( R \). Choose \( r \in C \setminus J(C) \). Then, there is \( x \in C \) such that \( 1 + rx \) is not a unit in \( C \), hence not a unit in \( R \). Thus, \( r \not\in J(R) \). Since \( R \) is generalized \( \sqrt{J} \)-fine, we can write \( r = u + a \) with \( u \in U(R) \) and \( a \in \sqrt{J(R)} \). Then, \( r = u(1 + u^{-1}a) \). Because \( r \) is central, we get \( ua = au \), and so \[ 1 + u^{-1}a \in 1 + \sqrt{J(R)} \subseteq U(R) \] knowing \cite[Lemma 2.1(1)]{DDE}. Hence, \( r \) is a unit in \( R \), and therefore a unit in \( C \). Thus, \( C \) is local.
\end{proof}

We now directly extract the following consequence.

\begin{corollary}\label{2.9}
Let \( R \) be a commutative ring and \( n \ge 1 \). Then, \( M_n(R) \) is a generalized $\sqrt{J}$-fine ring if and only if \( R \) is local. In particular, a commutative ring is generalized $\sqrt{J}$-fine if and only if it is local.
\end{corollary}

\begin{proof}
The ``if'' part follows from Theorem \ref{22}. For the ``only if'' part, note that \( C(M_n(R)) \cong R \), and so the claim follows immediately from Lemma \ref{2.8}.
\end{proof}

Mimicking the traditional terminology, we say that a group \( G \) is a \textit{\( p \)-group} if the order of every element of \( G \) is a power of the prime number \( p \). Moreover, a group \( G \) is said to be \textit{locally finite} if every finitely generated subgroup is finite.

Suppose now that \( G \) is an arbitrary group and \( R \) is an arbitrary ring. As usual, the symbol \( RG \) stands for the group ring of \( G \) over \( R \). The homomorphism \( \varepsilon : RG \to R \), defined by
\[\varepsilon \left( \sum_{g \in G} a_g g \right) = \sum_{g \in G} a_g,\]
is called the augmentation map of \( RG \), and its kernel, denoted by \( \Delta(RG) \), is called the \textit{augmentation ideal} of \( RG \).

\medskip

Our final result is the following criterion.

\begin{theorem}\label{111}
Let \(R\) be a ring and \(G\) a group. Then, the following two conditions hold:

(1) If \(RG\) is a generalized \(\sqrt{J}\)-fine ring, then \(R\) is a generalized \(\sqrt{J}\)-fine ring and \(G\) is a \(p\)-group with \(p \in J(R)\).

(2) The converse of (1) holds if \(G\) is a locally finite group.
\end{theorem}

\begin{proof}
(1) First, we show that \(R\) is a generalized \(\sqrt{J}\)-fine ring. In fact, let \(a \not\in J(R)\). Then, \(a \not\in J(RG)\) since \(J(RG) \cap R \subseteq J(R)\). Hence, there exist \(u \in U(RG)\) and \(t \in \sqrt{J(RG)}\) such that \(a = u + t\). Thus, we have
\[
a = \varepsilon(a) = \varepsilon(u) + \varepsilon(t) \in U(R) + \sqrt{J(R)}.
\]

Now, we establish that \(G\) is a \(p\)-group with \(p \in J(R)\). Indeed, working with \cite[Proposition 15(i)]{con}, it suffices to show that \(\Delta(RG) \subseteq J(RG)\). In this vein, suppose \(f \in \Delta(RG) \setminus J(RG)\). Since \(RG\) is a generalized \(\sqrt{J}\)-fine ring, there exist \(u \in U(RG)\) and \(t \in \sqrt{J(RG)}\) such that \(f = u + t\). Hence, it must be that
\[
0 = \varepsilon(f) = \varepsilon(u) + \varepsilon(t),
\]
so that \[\varepsilon(u) = -\varepsilon(t) \in U(RG) \cap \sqrt{J(RG)},\] which is a contradiction. Therefore, \(\Delta(RG) \subseteq J(RG)\).

(2) Assume that \(R\) is a generalized \(\sqrt{J}\)-fine ring and \(G\) is a locally finite \(p\)-group with \(p \in J(R)\). Since \(RG/\Delta(RG) \cong R\), and \(R\) is a generalized \(\sqrt{J}\)-fine ring, we conclude that \(RG/\Delta(RG)\) is also generalized \(\sqrt{J}\)-fine. 

On the other side, \cite[Lemma 2]{zhou} informs us that \(\Delta(RG) \subseteq J(RG)\). Hence, Lemma \ref{2.10} tells us that \(RG\) is a generalized \(\sqrt{J}\)-fine ring, as promised.
\end{proof}

\medskip
\medskip

\noindent{\bf Funding:} This scientific work is mainly based upon research funded by Iran National Science Foundation (INSF) under Project no. 40502582.

\medskip


\medskip

\section*{Data availability}

\textbf{No} data was used for the research described in the article.

\medskip

\section*{Declarations}

The authors declare \textbf{no} any conflict of interests while writing and preparing this manuscript.

\end{document}